\documentclass[11pt,reqno]{article}
\usepackage{amssymb}
\usepackage{amsmath}
\usepackage{amsthm}
\usepackage{amsfonts}
\usepackage{url}
\usepackage{breakurl}
\usepackage{hyperref}
\usepackage{comment}

\renewcommand{\ge}{\geqslant}

\newcommand{\C}{{\mathbb C}}

\newcommand{\primes}{{\mathcal P}}

\theoremstyle{plain}

\theoremstyle{definition}

\theoremstyle{claim}
\newtheorem{claim}{Claim}
\newtheorem{remark}{Remark}

\begin{document}
\bibliographystyle{plain}
\title{On some results of Ag\'elas concerning the GRH\\
and of Vassilev-Missana concerning the\\ prime zeta function
}
\author
{Richard P.\ Brent\footnote{MSI, Australian National University,
Canberra, Australia;
also CARMA, University of Newcastle, Australia;
email {\tt <GRH@rpbrent.com>}} 
}
\maketitle

\begin{abstract}
A recent paper by Ag\'elas [\emph{Generalized Riemann Hypothesis}, 2019,
hal-00747680v3] claims to prove the
Generalized Riemann Hypothesis (GRH)
and, as a special case, the Riemann Hypothesis (RH). 
We show that the proof given
by Ag\'elas contains an error. In particular, Lemma~2.3 of
Ag\'elas is false. This Lemma~2.3 is a generalisation of Theorem~1 of
Vassilev-Missana [\emph{A note on prime zeta function and 
Riemann zeta function},
Notes on Number Theory and Discrete Mathematics, 
22, 4 (2016), 12--15]. We show by several independent methods that
Theorem~1 of Vassilev-Missana is false. 
We also show that Theorem~2 of Vassilev-Missana is false.

This note has two aims.
The first aim is to alert other researchers to these errors so they do 
not rely on faulty results in their own work.  
The second aim is pedagogical~--- we hope to show how these
errors could have been detected earlier, 
which may suggest how similar errors can be avoided, 
or at least detected at an early stage.
\end{abstract}

\vspace*{\fill}
\pagebreak[4]

\section{Introduction}				\label{sec:Intro}

In~\cite{Agelas}, Ag\'elas states\footnote{We use the word ``Claim''
for a statement which we may later prove to be false.}

\begin{claim}[Ag\'elas, Theorem 2.1]		\label{claim:1}
For any Dirichlet character $\chi$ modula $k$, the Dirichlet L-function
$L(\chi, s)$ has all its non-trivial zeros on the critical line
$\Re(s) = \frac12$.
\end{claim}
This is the \emph{Generalized Riemann Hypothesis},
probably formulated by Adolf Piltz in 1884 
(see Davenport~\cite[p.~124]{Davenport}).	
A special case, which corresponds to the principal character
$\chi_0(n) = 1$ and the Riemann zeta-function $\zeta(s)$,
is the well-known \emph{Riemann Hypothesis}~\cite{Sarnak}.

Ag\'elas defines the half-plane
${\mathcal A} := \{s\in\C : \Re(s) > 1\}$, and
two Dirichlet series (convergent for $s \in \mathcal A$)
\[ 
P(\chi,s) := \sum_{p\in\primes}\chi(p)p^{-s}
\]
and
\[ 
P_2(\chi,s) := \sum_{p\in\primes}\chi(p)^2p^{-s},
\]
where $\primes$ is the set of primes $\{2, 3, 5, \ldots\}$.

When trying to understand the proof of Claim~\ref{claim:1} 
by Ag\'elas, we considered the case of the Riemann zeta-function.
Since this was sufficient to find an error in the proof, we
only need to consider this case. Thus we can take $\chi(p) = 1$, 
so $P(\chi,s)$ and $P_2(\chi,s)$ both reduce to
the usual \emph{prime zeta function}~\cite{Froberg}
\[
P(s) := \sum_{p\in\primes}p^{-s}.
\]
This function is also considered by
Vassilev-Missana~\cite{Vassilev-Missana}.
It is well-known (and a proof may be found in~\cite[p.~188]{Froberg}) that, 
for $\Re(s) > 1$,
\begin{equation}				\label{eq:P_Mobius}
P(s) = \sum_{k=1}^\infty \frac{\mu(k)}{k}\log\zeta(ks).
\end{equation}

Vassilev-Missana states
\begin{claim}[Vassilev-Missana, Theorem 1]		\label{claim:2}
For integer%
\footnote{It is not clear why Vassilev-Missana imposes such a
strong restriction on $s$; we might expect the relation to hold
for all $s \in \mathcal A$ or perhaps (using analytic continuation)
for almost all $s \in \{z\in\C: \Re(z) > 0\}$.}
$s > 1$, the relation
\[
(1-P(s))^2 = \frac{2}{\zeta(s)} - 1 + P(2s)
\;\;\text{ holds.}
\]
\end{claim}

Ag\'elas states 
\begin{quotation}
``Lemma 2.3 appears as an extension of Theorem 1 of Vassilev-Missana (2016),
we give here the details of the proof as it is at the heart of the Theorem
obtained in this paper. For this, we borrow the arguments used in 
Vassilev-Missana (2016).''
\end{quotation}
\pagebreak[3]

\noindent He then states
\begin{claim}[Ag\'elas, Lemma 2.3]			\label{claim:3}
For $s \in \mathcal A$, we have
\[
(1-P(\chi,s))^2L(\chi,s) - (P_2(\chi,2s)-1)L(\chi,s) = 2.
\]
\end{claim}
In the case that we consider, namely $L(\chi,s) = \zeta(s)$,
both Claim~\ref{claim:2} and Claim~\ref{claim:3} amount to the same
relation, which we can write in an equivalent form as
\begin{equation}					\label{eq:the_claim}
\frac{2}{\zeta(s)} = 2 - 2P(s) + (P(s))^2 - P(2s).
\end{equation}

In \S\ref{sec:disproof} we
show that~\eqref{eq:the_claim} is false.
This implies that Lemma~2.3 of Ag\'elas is false, as is Theorem~1 of
Vassilev-Missana.  Theorem~2.1 of Ag\'elas (the GRH) may be true,
but has not been proved. Lemmas~2.4 and 2.5 of Ag\'elas depend on 
Lemma~2.3, so are most likely false.
Theorem~2 of Vassilev-Missana is also false, 
as discussed in~\S\ref{sec:Thm2}.

\section{Disproving (\ref{eq:the_claim})}       \label{sec:disproof}

We give several methods to disprove~\eqref{eq:the_claim}.\\

\noindent{\bf Method 1.}
Expand each side of~\eqref{eq:the_claim} as a Dirichlet series
$\sum a_n n^{-s}$.
On the right-hand side (RHS), the only terms with nonzero coefficients $a_n$
are for integers $n$ of the form $p^\alpha q^\beta$, where $p$ and $q$ are
primes, $\alpha \ge 0$, and $\beta \ge 0$.  However, on the left-hand side 
(LHS), we find $a_{30} = -2 \ne 0$, since $30 = 2\times 3\times 5$ has
three distinct prime factors, implying that $\mu(30) = -1$.
By the uniqueness of Dirichlet series that converge absolutely for all
sufficiently large values of $\Re(s)$ \cite[Thm.~4.8]{Hildebrand},
we have a contradiction, so~\eqref{eq:the_claim} is false. 
\qed
\begin{remark}
{\rm
Instead of $30$ we could take any squarefree positive integer 
with greater than two prime factors.  A somewhat analogous situation arises
when bounding the coefficients of the cyclotomic polynomial
$\Phi_n(x)$ (see for example \cite{Maier,Riesel}).
Migotti~\cite{Migotti} showed that if $n$ has at most two distinct odd 
prime factors, then the coefficients of 
$\Phi_{n}(x)$ are all in the set $\{0,-1,1\}$. 
However, this is not necessarily true if $n$ has greater than two distinct
odd prime factors.  For example, $\Phi_{105}(x)$ 
contains terms $-2x^7$ and $-2x^{41}$.
\qed	
} 
\end{remark}

\noindent{\bf Method 2.} 
We can evaluate both sides of~\eqref{eq:the_claim} numerically
for one or more convenient values of~$s$.
If we take $s = 2k$ for some positive integer $k$, then
the LHS of~\eqref{eq:the_claim}
can easily be evaluated using Euler's formula
\[\zeta(2k) = \frac{(-1)^{k-1}(2\pi)^{2k}}{2\cdot (2k)!}\,B_{2k}\,,\]
where $B_{2k}$ is a Bernoulli number.
The RHS can be evaluated by using~\eqref{eq:P_Mobius}.
Taking $k=1$, i.e.\ $s=2$, the LHS is
$12/\pi^2 = 1.2158542$ and the RHS is $1.2230397$ 
(both values correct to $7$ decimals). 
Thus, $|\text{LHS}-\text{RHS}| > 0.007$. 
This is a contradiction, so~\eqref{eq:the_claim} is false.	\qed
\begin{remark}
{\rm
It is always a good idea to verify identities numerically whenever it is
convenient to do so. A surprising number of typos and more serious errors
can be found in this manner. Early mathematicians such as Euler, Gauss,
and Riemann were well aware of the value of numerical computation, even
though they lacked the electronic tools and mathematical software
that we have today.

If we had followed the philosophy of 
``experimental mathematics''~\cite{Borwein}, we would have attempted
method~2 first. It is only because we were familiar with the computation
of cyclotomic polynomials that we thought of using method~1 before
trying method~2.						\qed
} 
\end{remark}

\noindent{\bf Method 3.}
We consider the behaviour of each side of~\eqref{eq:the_claim} near
$s = 1$. On the LHS there is a simple zero at $s=1$, since the denominator
$\zeta(s)$ has a simple pole. On the RHS there is a logarithmic
singularity of the form
$a(\log(s-1))^2 + b\log(s-1) + O(1)$. 
This is a contradiction, so~\eqref{eq:the_claim} is false.	\qed\\

\noindent{\bf Method 4.}
For this method we need to assume that both sides of~\eqref{eq:the_claim}
have been extended by analytic continuation into the critical strip.
The LHS of~\eqref{eq:the_claim} has singularities precisely where
$\zeta(s)$ has zeros.
Using~\eqref{eq:P_Mobius}, $P(s)$ has singularities at these zeros
(say $\rho$), and also at $\rho/2$, $\rho/3, \ldots$.
Thus, on the line $\Re(s) = 1/2$, the LHS has infinitely many 
simple poles\footnote{This is 
because $\zeta(s)$ has infinitely many simple zeros 
on the critical line~\cite{Heath-Brown}.}
and no logarithmic singularities,
whereas the RHS has infinitely many logarithmic singularities.
This is a contradiction, so~\eqref{eq:the_claim} is false.	\qed\\

\noindent{\bf Method 5.}
The LHS of~\eqref{eq:the_claim} has a meromorphic continuation to the
whole of the complex plane, with poles wherever $\zeta(s)$ has zeros.
On the other hand, it is known that $P(s)$ has a natural boundary
at the line $\Re(s) = 0$, and by extending the argument of Method~4
we can show that the RHS of~\eqref{eq:the_claim} also has a natural
boundary at this line. Again, this is a contradiction. 
\qed

\pagebreak[3]

\section{Theorem 2 of Vassilev-Missana is false}	\label{sec:Thm2}

Vassilev-Missana~\cite[Theorem 2]{Vassilev-Missana} makes the following
claim.
\begin{claim}					\label{claim:4}
For integer\footnote{As before, it is not clear why the integer
restriction needs to be imposed.}
 $s > 1$,
\begin{equation}				\label{eq:thm2}
P(s) =
1-\sqrt{2/\zeta(s) - \sqrt{2/\zeta(2s) - \sqrt{2/\zeta(4s) - 
  \sqrt{2/\zeta(8s) - \cdots}}}}
\end{equation}
\end{claim}

\begin{proof}[Proof that Claim $\ref{claim:4}$ is incorrect]
Assume that Claim~$\ref{claim:4}$ is correct.
Replacing $s$ by $2s$ and using the result
to simplify~\eqref{eq:thm2}, we obtain
\begin{equation}			\label{eq:reverse-eng}
1 - P(s) = \sqrt{2/\zeta(s) - (1-P(2s))}.
\end{equation}
Squaring both sides of~\eqref{eq:reverse-eng} and simplifying 
gives~\eqref{eq:the_claim}, but we showed in \S\ref{sec:disproof}
that~\eqref{eq:the_claim} is incorrect. This contradiction shows
that~Claim~$\ref{claim:4}$ is incorrect.
\end{proof}

\begin{remark}
{\rm
An alternative is to resort to a variation on method~2 above.
With $s=2$
we find numerically that 	
$P(s) \approx 0.4522$		
and 
\[
1-\sqrt{2/\zeta(s) - \sqrt{2/\zeta(2s) - \sqrt{2/\zeta(4s) - \cdots}}}
\approx 0.4588 \ne P(s),	
\]
where the numerical values are correct to $4$ 
decimal places. Thus,~\eqref{eq:thm2} is incorrect.
} 
\end{remark}
\begin{remark}					\label{remark:defn}
{\rm
It may not be clear what the infinite expression
on the RHS of~\eqref{eq:thm2} means. 
We state Claim~\ref{claim:4} more precisely
as
\begin{equation}				\label{eq:precise}			
P(s) = 
1- \lim_{n\to\infty}
   \sqrt{2/\zeta(s) - \sqrt{2/\zeta(2s) - \sqrt{2/\zeta(4s) - \cdots 
  \sqrt{2/\zeta(2^n s)}}}}\;.
\end{equation}
The limit exists and is real if $s$ is real, positive, and sufficiently
large.

To evaluate~\eqref{eq:precise} numerically, we start with
a sufficiently large value of $n$, then evaluate the nested
square roots 
in~\eqref{eq:precise} by working from right to left, using the
values of $\zeta(2^n s), \zeta(2^{n-1}s), \ldots, \zeta(2s), \zeta(s)$.
\qed
} 
\end{remark}

\pagebreak[3]

\pagebreak[4]
\section*{Postscript (21 March 2021)}

Kannan Soundararajan kindly pointed out that there has been some discussion
of the paper by Vassilev-Missana~\cite{Vassilev-Missana} on the
MathOverflow website. For this, see
\url{https://mathoverflow.net/questions/288847/}.

In particular, the anonymous user Lucia pointed out the Dirichlet series
argument (see Method~1 above) for disproving~\eqref{eq:the_claim}.
Despite this, we have not found any erratum or retraction by the
author of~\cite{Vassilev-Missana}.

Since the comments by user Klangen on MathOverflow 
regarding the numerical investigation of Claim~\ref{claim:4}
are inconclusive,
we mention that, for the accurate numerical evaluation
of~\eqref{eq:precise}, it is desirable to use
\begin{equation}				\label{eq:precise2}	
P(s) = 
1- \lim_{n\to\infty}
   \sqrt{2/\zeta(s) - \sqrt{2/\zeta(2s) - \sqrt{2/\zeta(4s) - \cdots 
  \sqrt{2/\zeta(2^n s)-1}}}}\;,
\end{equation}
since this has the same limit,
but converges faster as $n \to \infty$.
An explanation
is that, when $n$ is large, we have
$\zeta(2^n s) = 1 + O(2^{-2^n s}) \approx 1$, so the ``tail'' of the 
expression~\eqref{eq:thm2} is approximately
\[
X := \sqrt{2-\sqrt{2-\sqrt{2-\cdots}}}\;\;.
\]
By squaring we see that $X$ satisfies the quadratic equation
$X^2 = 2-X$, whose only positive real root is $X=1$.
Thus, it is better to approximate the tail by $1$, as
in~\eqref{eq:precise2}, than by $0$, as in~\eqref{eq:precise}.
\end{document}